\documentclass[11pt,oneside,reqno]{amsart}
\usepackage{amssymb, amscd, amsmath, amsthm, latexsym, hyperref,xcolor,gensymb}
\usepackage{pb-diagram, fancyhdr, graphicx, psfrag}
\usepackage[text={6.3in,8.5in},centering,letterpaper,dvips]{geometry}
\usepackage[
textwidth=0.9in,
backgroundcolor=yellow,
bordercolor=orange,
textsize=scriptsize]{todonotes}
\usepackage{tikz}
\usetikzlibrary{calc}
\usepackage{standalone}
\usepackage{array,mathtools}
\usepackage{cite, accents}

\def\cs{\mathop{\cs }}

\def\calf{\mathcal{F}}

\def\calc{\mathcal{C}}

\def\call{\mathcal{L}}

\def\calb{\mathcal {B}}

\def\cfk{{\textrm{CFK}}}

\newcommand{\spinc}{\ifmmode{{\mathfrak s}}\else{${\mathfrak s}$\ }\fi}
\newcommand{\spinct}{\ifmmode{{\mathfrak t}}\else{${\mathfrak t}$\ }\fi}
\newcommand{\spincw}{\ifmmode{{\mathfrak w}}\else{${\mathfrak w}$\ }\fi}
\def\Z{\mathbb Z}

\def\R{\mathbb R}

\def\F{\mathbb F}

\def\alg{{\rm alg}}
\def\alex{{\rm Alex}}

\def\cfki{\cfk ^\infty(K)}
\def\cs{\, \#  \,}

\theoremstyle{plain}

\newtheorem{theorem}{Theorem}[section]

\newtheorem*{theorem*}{Theorem}

\newtheorem{corollary}[theorem]{Corollary}

\newtheorem{proposition}[theorem]{Proposition}

\newtheorem{conjecture}[theorem]{Conjecture}
\theoremstyle{definition}
\newtheorem{definition}[theorem]{Definition}

\theoremstyle{remark}

\numberwithin{equation}{section}

\def\U{\Upsilon}

\begin{document}


\title[Stable Equivalence of Connected Sums of Torus Knots]{Using Secondary Upsilon Invariants to Rule out stable equivalence of knot complexes}
\author{Samantha Allen}
\address{Samantha Allen: Department of Mathematics, Indiana University, Bloomington, IN 47405 }
\email{allensam@indiana.edu}

\begin{abstract}  Two Heegaard Floer knot complexes are called stably equivalent if an acyclic complex can be added to each complex to make them filtered chain homotopy equivalent.  Hom showed that if two knots are concordant, then their knot complexes are stably equivalent.  Invariants of stable equivalence include the concordance invariants $\tau$, $\varepsilon$, and $\Upsilon$.  Feller and Krcatovich gave a relationship between the Upsilon invariants of torus knots.  We use secondary Upsilon invariants  defined by Kim and Livingston to show that these relations do not extend to stable equivalence.
\end{abstract}

\maketitle



\section{Introduction}
In general, the study of torus knots and their concordance invariants has been a frequent topic of investigation.  One early highlight was Litherland's proof of the independence of torus knots using Tristram-Levine signature functions in \cite{Litherland}.  Because of their role in studying algebraic curves, research on invariants of torus knots continues.  In particular, the Ozsv\'ath--Stipsicz--Szab\'o Upsilon function has been used in \cite{borodzik-hedden} and \cite{feller}.  Recently, Feller and Krcatovich (in \cite{feller-krcatovich}) determined relationships among the Upsilon functions of torus knots.  Our goal here is to use the secondary Upsilon invariants, defined by Kim and Livingston in \cite{kim-livingston}, to show that these relationships do not extend to stabilized knot complexes of torus knots.  

Two Heegaard Floer knot complexes are called stably equivalent if an acyclic complex can be added to each complex to make them filtered chain homotopy equivalent.  In \cite{hom2}, Hom showed that if two knots are concordant, then their knot complexes are stably equivalent.  The concordance invariants $\tau$, $\varepsilon$, $\Upsilon$, $\U^2$ are all invariants of the stable equivalence class (see \cite{hom1, hom2, oss, os2, kim-livingston}).  We give an example of a pair of torus knots which have identical Upsilon invariants (by Feller and Krcatovich \cite{feller-krcatovich}) but differing secondary Upsilon invariants, and thus have knot complexes which are not stably equivalent.


\begin{theorem}\label{main thm}
The knot complex $\cfk ^{\infty}(T(5,7))$ is not stably equivalent to the knot complex $\cfk ^{\infty}(T(2,5)\cs T(5,6))$.
\end{theorem}

 \noindent Similar procedures show that $\cfk ^{\infty}(T(7,9))$ is not stably equivalent to $\cfk ^{\infty}(T(2,7)\cs T(7,8))$, and, in fact, the following general theorem holds.
\begin{theorem}\label{general thm}
For all $p\geq 5$ odd, the knot complex $\cfk ^{\infty}(T(p, \, p+2))$ is not stably equivalent to $\cfk ^{\infty}(T(2,p)\cs T(p, \, p+1))$.
\end{theorem}

Furthermore, in their original paper introducing the secondary {\it Upsilon} invariants, Kim and Livingston \cite{kim-livingston} showed that $\U ^2$ is stronger than $\U$ for a single pair of knots, as well as a family of complexes which are not known to be knot complexes.  We give the first example of a family of knots for which $\U ^2$ is stronger.


\section{Knot complexes, $\cfki $} \label{basics}
To each knot $K \subset S^3$, we can associate a  chain complex $\cfki$ (see \cite{os2}).  It is equipped with a grading (called the {\it Maslov} grading) having the property that the boundary map decreases gradings by 1. The complex $\cfki$ is also bifiltered --- each element $x$ has an {\it algebraic} and an {\it Alexander} filtration, denoted by $\alg (x)$ and $\alex (x)$ respectively.  We consider these complexes up to bifiltered chain homotopy equivalence, denoted here by $\simeq$.    

We can represent $\cfki$ as a diagram in the $(\alg,\alex)$--plane, as in Figure ~\ref{fig34-labeled}.  Let $\calb$ be a bifiltered basis for $\cfki$.  Then each element $x\in \calb$ is represented by the point $(\alg (x), \alex (x))$ and the boundary map is indicated by arrows; for example, $\partial (b) = a+c$.  (We discuss the case when two basis elements have the same filtration levels below.)  Throughout this paper, when it will cause no confusion, we will use the $(i,j)$ coordinates interchangeably with the basis element at those filtration levels.  Here white dots represent elements at grading 0 and black dots represent elements at grading 1.  

\begin{figure}[ht]
  \centering
  \resizebox{4cm}{4cm}{     \begin{tikzpicture}
        \coordinate (Origin)   at (0,0);
        \coordinate (XAxisMin) at (-1,0);
        \coordinate (XAxisMax) at (4,0);
        \coordinate (YAxisMin) at (0,-1);
        \coordinate (YAxisMax) at (0,4);
        \draw [thin, black,-latex] (XAxisMin) -- (XAxisMax);
        \draw [thin, black,-latex] (YAxisMin) -- (YAxisMax);
        \clip (-1,-1) rectangle (4,4);
        \foreach \x in {1,...,3}
     		\draw (\x,1pt) -- (\x,-3pt)
			node[anchor=north] {\x};
    	\foreach \y in {1,...,3}
     		\draw (1pt,\y) -- (-3pt,\y) 
     			node[anchor=east] {\y};

        \node[draw,circle,fill=white,scale=0.8,label={[xshift=0.2cm, yshift=-0.1cm]a}] (A) at (0,3) {};
        \node[draw,circle,fill,scale=0.4,label={[xshift=0.2cm, yshift=-0.1cm]b}] (B) at (1,3) {b};
        \node[draw,circle,fill=white,scale=0.8,label={[xshift=-0.25cm, yshift=-0.5cm]c}] (C) at (1,1) {};
        \node[draw,circle,fill,scale=0.4,label={[xshift=0.2cm, yshift=-0.1cm]d}] (D) at (3,1) {d};
        \node[draw,circle,fill=white,scale=0.8,label={[xshift=0.3cm, yshift=-0.15cm]e}] (E) at (3,0) {};

        \path [very thick] [->] (B) edge node[left] {} (A);
        \path [very thick] [->] (B) edge node[left] {} (C);
        \path [very thick] [->] (D) edge node[left] {} (C);
        \path [very thick] [->] (D) edge node[left] {} (E);

    \end{tikzpicture}}
  \caption{The knot complex for the torus knot $T(3,4)$, $\cfk ^{\infty}(T(3,4))$}
  \label{fig34-labeled}
\end{figure}
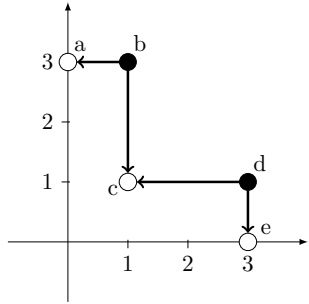

The chain complex $\cfki$ also has a compatible $\F[U,U^{-1}]$ structure, where $\F$ is the field of two elements.  The action of $U$ decreases both filtration levels by 1 and the Maslov grading by 2.  To form the diagram of the full complex, we take all integer diagonal translates of the diagram shown.  Unless we need to use the $U$-action, we will hide this structure in any diagrams.

In general, for a given knot $K$, $\cfki$ may have multiple elements at the same filtration levels.  In this case, we use a grid to show the complex and each bifiltered basis element at filtration level $(i,j)$ will be shown in the unit square above and to the right of the point $(i,j)$.  For $L$--space knots, however, the complex $\cfki$ is always a {\it staircase} complex, as in Figure ~\ref{fig34-labeled}.    In this case, the height and width of each step is determined by the gaps in the exponents of the Alexander polynomial of $K$; the Alexander polynomial can be written as 
$$\Delta_K(t)=\sum_{i=0}^d (-1)^i t^{a_i}$$
for some $\{a_i\}$ and $\cfki$ is a staircase of the form $$[a_1-a_0, a_2-a_1,...,a_d-a_{d-1}]$$ where the indices alternate between horizontal and vertical steps.  For more details, see \cite{os3} and \cite{borodzik-livingston}.  For example, for $K=T(3,4)$, 
$$\Delta_K(t)=1-t+t^3-t^5+t^6,$$
so $\cfki$ is of the form $$[1-0, 3-1, 5-3, 6-5] = [1,2,2,1]$$ as shown.

For use in later sections, we record some properties of the complex $\cfki$.
\begin{theorem}[\hspace{0.1pt}\cite{os1, hom2}] For knots $K,J\subset S^3$, 
\begin{enumerate}
\item $\cfk ^{\infty}(K)\otimes \cfk ^{\infty}(J)\simeq\cfk ^{\infty}(K\cs J)$.
\item $\cfk ^{\infty}(-K)\simeq \cfk ^{\infty}(K)^*$.  In terms of the diagram, $\cfk ^{\infty}(K)^*$ is obtained from $\cfk ^{\infty}(K)$ via a $180\degree$ rotation (each bifiltered basis element $(i,j)$ in $\cfk ^{\infty}(K)$ is represented by $(-i,-j)$ in $\cfk ^{\infty}(-K)$) and reversing all arrows.
\item if $K$ and $J$ are concordant knots, then there are acyclic complexes $A_1$ and $A_2$ such that $\cfk ^{\infty}(K)\oplus A_1 \simeq \cfk ^{\infty}(J)\oplus A_2$.
\end{enumerate}
\end{theorem}
\noindent The third property leads to the following definition.
\begin{definition} (\hspace{0.1pt}\cite{hom2}).
Two knot complexes $\cfk ^{\infty}(K_1)$ and $\cfk ^{\infty}(K_2)$ are called {\it stably equivalent} if there are acyclic complexes $A_1$ and $A_2$ such that $\cfk ^{\infty}(K_1)\oplus A_1 \simeq \cfk ^{\infty}(K_2)\oplus A_2$.
\end{definition}
\noindent See, for instance, \cite{os1} for a more detailed description of the $\cfki$ complex and \cite{hom2} for more discussion on stable equivalence.


\section{The Upsilon Invariant}
For a knot $K$ and $t\in [0,2]$, a filtration can be defined on $\cfki$ by the function
$$\frac{t}{2}\,\alex (x)+\left(1-\frac{t}{2}\right)\alg (x).$$
  Based on this filtration, we define a family of subcomplexes of $\calf_{t,s} \subset \cfki$ by
$$\calf_{t,s}=\left\{x\in \calb \left| \left(\frac{t}{2}\alex (x)+\left(1-\frac{t}{2}\right)\alg (x)\right)\leq s\right\}\right.$$
for $t\in [0,2]$ and $s\in\R$ where $\calb$ is a bifiltered basis for $\cfki$.  The subcomplex is independent of the choice of basis. Diagrammatically, the subcomplex $\calf_{t,s}$ is represented as a half-space with boundary line 
$$\frac{t}{2}j+\left(1-\frac{t}{2}\right)i = s$$
which has slope $m=1-\frac{2}{t}$ and $j$--intercept $b=\frac{2s}{t}$.  We call this boundary line the {\it support line} and denote it by $\call_{t,s}$.  
\begin{definition}
Let
$$\gamma_K(t) = \text{min} \left\{s\, |\, H_0(F_{t,s}) \longrightarrow H_0(\cfki) \text{ is surjective} \right\}.$$
In \cite{oss}, Ozsv\'ath, Stipsicz, and Szab\'o define the knot invariant Upsilon $\U_K(t)$ for $t\in[0,2]$.  In \cite{livingston1}, it is shown that $\U_K(t)=-2\cdot \gamma_K(t).$
\end{definition}
\begin{theorem}[As in \cite{oss}]
For knots $K,J \subset S^3$,
\begin{enumerate}
\item $\U_K(t)$ is piecewise linear.
\item $\U_{K\cs  J}(t) = \U_{K}(t)+\U_{J}(t)$.
\item $\U_{-K}(t)=-\U_{K}(t)$.
\item If $K$ is slice, $\U_K(t)=0.$
\end{enumerate}
\end{theorem}
\noindent Based on these properties, we get the following corollary:
\begin{corollary}[\hspace{0.1pt}\cite{livingston1}]
If $K,J \subset S^3$ are concordant knots, then $\U_K(t)=\U_J(t)$.
\label{ups conc}
\end{corollary}
\noindent The {\it Upsilon} invariant is also a stable equivalence invariant.  Feller and Krcatovich gave the following relation.
\begin{theorem}[\hspace{0.1pt}\cite{feller-krcatovich}]\label{feller-krcatovich}
Let $p<q$ be coprime integers.  Then 
$$\U_{T(p, \, q)}(t)=\U_{T(p,\, q-p)}(t)+\U_{T(p, \, p+1)}(t).$$
\end{theorem}
\noindent Thus $\U$ cannot differentiate between the stable equivalence classes of $T(p,q)$ and 
\newline \noindent $T(p,q-p)\cs  T(p, \, p+1)$.


\section{Secondary Upsilon Invariants}
In \cite{kim-livingston}, Kim and Livingston defined the family of secondary {\it Upsilon} invariants $\U_{K,t}^{2}(s)$.  For our purposes, we may restrict to a situation where the definition is simple.  We will consider only knots $K$ such that $\Delta\U'_K(t)>0$ and we will define $\U_{K,t}^2(t)$ (removing the dependence on $s$ in the original definition) at $t$ which are singularities of $\U_K'(t)$.

Let $K\subset S^3$ and $t\in[0,2]$ and denote $$\calf_t:=\calf_{t, \gamma_K(t)}.$$
Let $t_0\in [0,2]$ be a singularity of $\U_K'(t)$.  If $\U_K'(t)>0$, then, for $\delta$ small enough, the set of cycles which are not boundaries $\mathcal{C}_{t_0}$ in $\calf_{t_0}$ is split into two disjoint sets $\mathcal{C}_{t_0+\delta}$ and $\mathcal{C}_{t_0-\delta}$; the sets of cycles which are not boundaries in $\calf_{t_0+\delta}$ and $\calf_{t_0-\delta}$ respectively.

\begin{definition}
For each $t_0\in [0,2]$ which is a singularity of $\U_K'(t)$, let
\begin{equation*}
\gamma_{K,t_0}^2(t_0)=\text{min} \left\{ r \left|
\begin{array}{c}
\exists \,x_{\pm}\in \calc _{t_{0} \pm \delta} \text{ such that } x_- \text{ and } x_+ \text{ represent }\\ \text{the same class in } H_0(\calf_{t_0}+\calf_{t_0,r})
\end{array} \right\}\right..
\end{equation*}
Then the secondary {\it Upsilon} invariants defined by Kim and Livingston \cite{kim-livingston} are given by 
$$\Upsilon_{K,t_0}^2(t_0) = -2\cdot (\gamma_{K,t_0}^2(t_0)-\gamma_{K}(t_0)).$$
\end{definition}
As an example,  in Figure ~\ref{fig34-ups2}, we see that $\U'_{T(3,4)}(t)$ has a singularity at $t_0 = \frac{2}{3}$.  Then $\calc_{t_0+\delta}=\{x_+\}$ and $\calc_{t_0-\delta}=\{x_-\}$ where $x_+$ and $x_-$ are represented by the points $(1,1)$ and $(0,3)$ respectively.  Let $z$ be the point represented by $(1,3)$.   We see that $\partial z= x_++x_-$, which implies $\U^2_{T(2,3),\frac{2}{3}}(\frac{2}{3})=-2\cdot \left(\frac{5}{3}-1\right)=-\frac{2}{3}$.
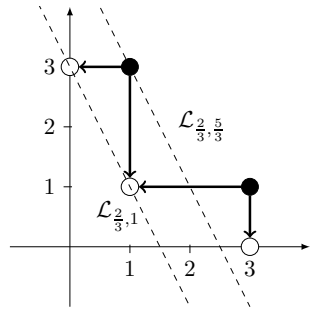
\begin{figure}[ht]
  \centering
  \resizebox{4cm}{4cm}{     \begin{tikzpicture}
        \coordinate (Origin)   at (0,0);
        \coordinate (XAxisMin) at (-1,0);
        \coordinate (XAxisMax) at (4,0);
        \coordinate (YAxisMin) at (0,-1);
        \coordinate (YAxisMax) at (0,4);
        \draw [thin, black,-latex] (XAxisMin) -- (XAxisMax);
        \draw [thin, black,-latex] (YAxisMin) -- (YAxisMax);
        \clip (-1,-1) rectangle (4,4);
              \foreach \x in {1,...,3}
     		\draw (\x,1pt) -- (\x,-3pt)
			node[anchor=north] {\x};
    	\foreach \y in {1,...,3}
     		\draw (1pt,\y) -- (-3pt,\y) 
     			node[anchor=east] {\y};

        \node[draw,circle,fill=white,scale=0.8] (A) at (0,3) {};
        \node[draw,circle,fill,scale=0.4] (B) at (1,3) {b};
        \node[draw,circle,fill=white,scale=0.8] (C) at (1,1) {};
        \node[draw,circle,fill,scale=0.4] (D) at (3,1) {d};
        \node[draw,circle,fill=white,scale=0.8] (E) at (3,0) {};

        \path [very thick] [->] (B) edge node[left] {} (A);
        \path [very thick] [->] (B) edge node[left] {} (C);
        \path [very thick] [->] (D) edge node[left] {} (C);
        \path [very thick] [->] (D) edge node[left] {} (E);
        \path [dashed] [-] (-1,5) edge node[label={[xshift=0.3cm, yshift=-2.0cm]$\mathcal{L}_{\frac{2}{3},1}$}] {} (2,-1);
        \path [dashed] [-] (-1,7) edge node[label={[xshift=1.2cm, yshift=-1.5cm]$\mathcal{L}_{\frac{2}{3},\frac{5}{3}}$}] {} (3,-1);

    \end{tikzpicture}}
  \caption{$\cfk ^{\infty}(T(3,4))$ with support lines for $t=\frac{2}{3}$.}
  \label{fig34-ups2}
\end{figure}

\begin{theorem}[\hspace{0.1pt}\cite{kim-livingston}]\label{ups2 conc}
$\U^2_{K,t}(s)$ is a stable equivalence invariant.
\end{theorem}
  
\section{Results}
We begin with a proof of Theorem ~\ref{main thm}, then use the same procedure to prove the general theorem.
\subsection{The case of $p=5$} 
$ $ \newline
To prove Theorem \ref{main thm}, we compute that when $t_0=\frac{4}{5}$\
$$\U^2_{T(5,7),t_0}(t_0)\neq \U^2_{T(2,5)\cs T(5,6),t_0}(t_0)$$ and then apply Theorem ~\ref{ups2 conc}.
\begin{proposition}
$\U^2_{T(5,7),\frac{4}{5}}(\frac{4}{5})=-\frac{8}{5}.$
\label{prop57}
\end{proposition}
\begin{proof}
As in Section ~\ref{basics}, we can compute $\cfk ^{\infty}(T(5,7))$ from the gaps in the exponents of the Alexander polynomial:
\begin{equation*}
\begin{array}{cc}
\Delta_{T(5,7)}(t) = & 
 1-t+t^5-t^6+t^7-t^8+t^{10}-t^{11}+t^{12}-t^{13}\\&+t^{14}-t^{16}+t^{17}-t^{18}+t^{19}-t^{23}+t^{24}.
\end{array}
\end{equation*}
So $\cfk ^{\infty}(T(5,7))$ is a staircase complex of the form
$$[1,4,1,1,1,2,1,1,1,1,2,1,1,1,4,1].$$
This yields a bifiltered graded basis $\calb$ for $\cfk ^{\infty}(T(5,7))$.  See Figure ~\ref{fig57} for the diagram for $\cfk ^{\infty}(T(5,7))$.

\begin{figure}[ht]
  \centering
  \resizebox{5cm}{5cm}{     \begin{tikzpicture}
        \coordinate (Origin)   at (0,0);
        \coordinate (XAxisMin) at (-1,0);
        \coordinate (XAxisMax) at (13,0);
        \coordinate (YAxisMin) at (0,-1);
        \coordinate (YAxisMax) at (0,13);
        \draw [thin, black,-latex] (XAxisMin) -- (XAxisMax);
        \draw [thin, black,-latex] (YAxisMin) -- (YAxisMax);
        \clip (-1,-1) rectangle (13,13);
               \foreach \x in {1,...,12}
     		\draw (\x,1pt) -- (\x,-3pt)
			node[anchor=north] {\x};
    	\foreach \y in {1,...,12}
     		\draw (1pt,\y) -- (-3pt,\y) 
     			node[anchor=east] {\y};

        \node[draw,circle,fill=white, scale=0.5] (A) at (0,12) {};
        \node[draw,circle,fill,scale=0.5] (B) at (1,12) {};
        \node[draw,circle,fill=white,scale=0.5] (C) at (1,8) {};
        \node[draw,circle,fill,scale=0.5] (D) at (2,8) {};
        \node[draw,circle,fill=white,scale=0.5] (E) at (2,7) {};
        \node[draw,circle,fill,scale=0.5] (F) at (3,7) {};
        \node[draw,circle,fill=white,scale=0.5] (G) at (3,5) {};
        \node[draw,circle,fill,scale=0.5] (H) at (4,5) {};
        \node[draw,circle,fill=white,scale=0.5] (I) at (4,4) {};
        \node[draw,circle,fill,scale=0.5] (J) at (5,4) {};
        \node[draw,circle,fill=white,scale=0.5] (K) at (5,3) {};
        \node[draw,circle,fill,scale=0.5] (L) at (7,3) {};
        \node[draw,circle,fill=white,scale=0.5] (M) at (7,2) {};
        \node[draw,circle,fill,scale=0.5] (N) at (8,2) {};
        \node[draw,circle,fill=white,scale=0.5] (O) at (8,1) {};
        \node[draw,circle,fill,scale=0.5] (P) at (12,1) {};
        \node[draw,circle,fill=white,scale=0.5] (Q) at (12,0) {};

        \path [very thick] [->] (B) edge node[left] {} (A);
        \path [very thick] [->] (B) edge node[left] {} (C);
        \path [very thick] [->] (D) edge node[left] {} (C);
        \path [very thick] [->] (D) edge node[left] {} (E);
        \path [very thick] [->] (F) edge node[left] {} (E);
        \path [very thick] [->] (F) edge node[left] {} (G);
        \path [very thick] [->] (H) edge node[left] {} (G);
        \path [very thick] [->] (H) edge node[left] {} (I);
        \path [very thick] [->] (J) edge node[left] {} (I);
        \path [very thick] [->] (J) edge node[left] {} (K);
        \path [very thick] [->] (L) edge node[left] {} (K);
        \path [very thick] [->] (L) edge node[left] {} (M);
        \path [very thick] [->] (N) edge node[left] {} (M);
        \path [very thick] [->] (N) edge node[left] {} (O);
        \path [very thick] [->] (P) edge node[left] {} (O);
        \path [very thick] [->] (P) edge node[left] {} (Q);


    \end{tikzpicture}}
  \caption{The knot complex for the torus knot $T(5,7)$, $\cfk ^{\infty}(T(5,7))$}
  \label{fig57}
\end{figure}
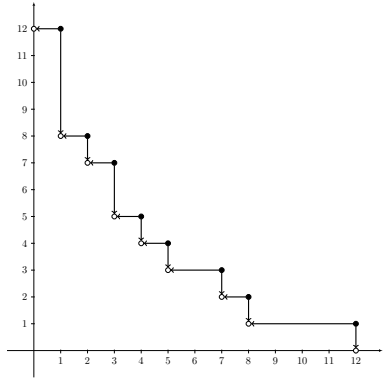

To prove the proposition, we first compute $\gamma\,_{T(5,7)}(\frac{4}{5})$.  Recall that 
$$\gamma\,_{T(5,7)}\left(\frac{4}{5}\right)=\text{min} \left\{s\, |\, H_0(F_{\frac{4}{5},s}) \longrightarrow H_0(\cfki) \text{ is surjective} \right\}.$$
So we need to find the minimal $s$ such that $\call_{\frac{4}{5},s}$ contains a bifiltered basis element (or multiple elements) in $\cfk ^{\infty}(T(5,7))$.  We compute that $\call_{\frac{4}{5},s}$ has slope $m=-\frac{3}{2}$ and  $j$--intercept $b=\frac{5s}{2}$.  In Figure ~\ref{fig57-ups2}, one can see that $\call_{\frac{4}{5},s}$ with minimal $s$ passes through the points $(1,8)$ and $(3,5)$.  The $j$-intercept of this line is $\frac{19}{2}$ corresponding to an $s$ value of $\frac{19}{5}$.  Thus $\gamma\,_{T(5,7)}(\frac{4}{5})=\frac{19}{5}$.  Note that near $t=\frac{4}{5}$, the line $\call_{t,s}$ pivots around the two points $(1,8)$ and $(3,5)$.  This causes a change in slope in $\U_K$ and so $t=\frac{4}{5}$ is a singulariy of $\U_K'$.

\begin{figure}[ht]
  \centering
  \resizebox{5cm}{5cm}{     \begin{tikzpicture}
        \coordinate (Origin)   at (0,0);
        \coordinate (XAxisMin) at (-1,0);
        \coordinate (XAxisMax) at (13,0);
        \coordinate (YAxisMin) at (0,-1);
        \coordinate (YAxisMax) at (0,13);
        \draw [thin, black,-latex] (XAxisMin) -- (XAxisMax);
        \draw [thin, black,-latex] (YAxisMin) -- (YAxisMax);
        \clip (-1,-1) rectangle (13,13);
               \foreach \x in {1,...,12}
     		\draw (\x,1pt) -- (\x,-3pt)
			node[anchor=north] {\x};
    	\foreach \y in {1,...,12}
     		\draw (1pt,\y) -- (-3pt,\y) 
     			node[anchor=east] {\y};

        \node[draw,circle,fill=white, scale=0.5] (A) at (0,12) {};
        \node[draw,circle,fill,scale=0.5] (B) at (1,12) {};
        \node[draw,circle,fill=white,scale=0.5] (C) at (1,8) {};
        \node[draw,circle,fill,scale=0.5] (D) at (2,8) {};
        \node[draw,circle,fill=white,scale=0.5] (E) at (2,7) {};
        \node[draw,circle,fill,scale=0.5] (F) at (3,7) {};
        \node[draw,circle,fill=white,scale=0.5] (G) at (3,5) {};
        \node[draw,circle,fill,scale=0.5] (H) at (4,5) {};
        \node[draw,circle,fill=white,scale=0.5] (I) at (4,4) {};
        \node[draw,circle,fill,scale=0.5] (J) at (5,4) {};
        \node[draw,circle,fill=white,scale=0.5] (K) at (5,3) {};
        \node[draw,circle,fill,scale=0.5] (L) at (7,3) {};
        \node[draw,circle,fill=white,scale=0.5] (M) at (7,2) {};
        \node[draw,circle,fill,scale=0.5] (N) at (8,2) {};
        \node[draw,circle,fill=white,scale=0.5] (O) at (8,1) {};
        \node[draw,circle,fill,scale=0.5] (P) at (12,1) {};
        \node[draw,circle,fill=white,scale=0.5] (Q) at (12,0) {};

        \path [very thick] [->] (B) edge node[left] {} (A);
        \path [very thick] [->] (B) edge node[left] {} (C);
        \path [very thick] [->] (D) edge node[left] {} (C);
        \path [very thick] [->] (D) edge node[left] {} (E);
        \path [very thick] [->] (F) edge node[left] {} (E);
        \path [very thick] [->] (F) edge node[left] {} (G);
        \path [very thick] [->] (H) edge node[left] {} (G);
        \path [very thick] [->] (H) edge node[left] {} (I);
        \path [very thick] [->] (J) edge node[left] {} (I);
        \path [very thick] [->] (J) edge node[left] {} (K);
        \path [very thick] [->] (L) edge node[left] {} (K);
        \path [very thick] [->] (L) edge node[left] {} (M);
        \path [very thick] [->] (N) edge node[left] {} (M);
        \path [very thick] [->] (N) edge node[left] {} (O);
        \path [very thick] [->] (P) edge node[left] {} (O);
        \path [very thick] [->] (P) edge node[left] {} (Q);
        \path [dashed, very thick] [-] (-1,11) edge node[label={[xshift=0.9cm, yshift=-4.0cm]{\Huge $\mathcal{L}_{\frac{4}{5},\frac{19}{5}}$}}] {} (7,-1);
        \path [dashed, very thick] [-] (-1,13) edge node[label={[xshift=0.9cm, yshift=-0.3cm]{\Huge $\mathcal{L}_{\frac{4}{5},\frac{23}{5}}$}}] {} (25/3,-1);

    \end{tikzpicture}}
  \caption{Support lines on $\cfk ^{\infty}(T(5,7))$ for $t=\frac{4}{5}$}
  \label{fig57-ups2}
\end{figure}
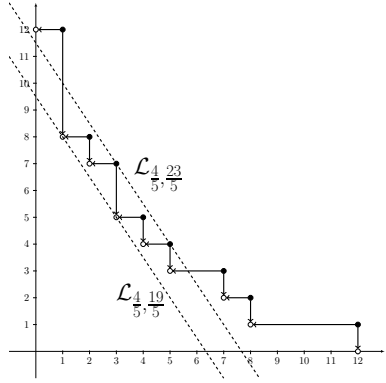

Now, we turn our attention to secondary Upsilon.  We have that
$$\calc_{\frac{4}{5}}=\{(1,8),(3,5)\},$$
and for small enough $\delta$
$$\calc_{\frac{4}{5}-\delta}=\{(1,8)\} \text{ and } \calc_{\frac{4}{5}+\delta}=\{(3,5)\}.$$
To determine  $\U^2_{T(5,7),\frac{4}{5}}(\frac{4}{5})$, we compute how far the line of slope $-\frac{3}{2}$ needs to be moved so that the elements represented by $(1,8)$ and $(3,5)$ are homologous in $\calf_{\frac{4}{5},r}$.  In the diagram we see that we need $\calf_{\frac{4}{5},r}$ to contain the elements represented by $(3,7)$ and $(2,8)$.  The minimal $r$ which accomplishes this is $r=\frac{23}{5}$, as shown in Figure ~\ref{fig57-ups2}.  So we have that $\gamma^2_{T(5,7),\frac{4}{5}}(\frac{4}{5}) = \frac{23}{5}$.  Thus 
$$\U^2_{T(5,7),\frac{4}{5}}(\frac{4}{5}) = -2\cdot \left(\frac{23}{5}-\frac{19}{5}\right) = -\frac{8}{5}.$$
\end{proof}
\begin{proposition}
$\U^2_{T(2,5)\cs T(5,6),\frac{4}{5}}(\frac{4}{5})<-\frac{8}{5}.$
\end{proposition}
\begin{proof}
We follow the same procedure as in the proof of Proposition ~\ref{prop57}.  The chain complexes $\cfk ^{\infty}(T(2,5))$ and $\cfk ^{\infty}(T(5,6))$ are both computed from their Alexander polynomials and then the tensor product is taken to produce $\cfk ^{\infty}(T(2,5)\cs T(5,6))$ as shown in Figures ~\ref{fig25+56-sep} and ~\ref{fig25+56-tensor}.

\begin{figure}[ht]
  \centering
\begin{tabular}{m{2.1in} m{2.1in}}
  \resizebox{1.2in}{1.2in}{     \begin{tikzpicture}
        \coordinate (Origin)   at (0,0);
        \coordinate (XAxisMin) at (-1,0);
        \coordinate (XAxisMax) at (3,0);
        \coordinate (YAxisMin) at (0,-1);
        \coordinate (YAxisMax) at (0,3);
        \draw [thin, black,-latex] (XAxisMin) -- (XAxisMax);
        \draw [thin, black,-latex] (YAxisMin) -- (YAxisMax);
        \clip (-1,-1) rectangle (3,3);
               \foreach \x in {1,...,2}
     		\draw (\x,1pt) -- (\x,-3pt)
			node[anchor=north] {\x};
    	\foreach \y in {1,...,2}
     		\draw (1pt,\y) -- (-3pt,\y) 
     			node[anchor=east] {\y};

        \node[draw,circle,fill=white, scale=0.5] (A) at (0,2) {};
        \node[draw,circle,fill,scale=0.5] (B) at (1,2) {};
        \node[draw,circle,fill=white,scale=0.5] (C) at (1,1) {};
        \node[draw,circle,fill,scale=0.5] (D) at (2,1) {};
        \node[draw,circle,fill=white,scale=0.5] (E) at (2,0) {};
       
        \path [very thick] [->] (B) edge node[left] {} (A);
        \path [very thick] [->] (B) edge node[left] {} (C);
        \path [very thick] [->] (D) edge node[left] {} (C);
        \path [very thick] [->] (D) edge node[left] {} (E);

    \end{tikzpicture}} &  \resizebox{2.0in}{2.0in}{     \begin{tikzpicture}
        \coordinate (Origin)   at (0,0);
        \coordinate (XAxisMin) at (-1,0);
        \coordinate (XAxisMax) at (11,0);
        \coordinate (YAxisMin) at (0,-1);
        \coordinate (YAxisMax) at (0,11);
        \draw [thin, black,-latex] (XAxisMin) -- (XAxisMax);
        \draw [thin, black,-latex] (YAxisMin) -- (YAxisMax);
        \clip (-1,-1) rectangle (11,11);
       
       \foreach \x in {0,...,10}
     		\draw (\x,1pt) -- (\x,-3pt)
			node[anchor=north] {\x};
    	\foreach \y in {0,...,10}
     		\draw (1pt,\y) -- (-3pt,\y) 
     			node[anchor=east] {\y};

        \node[draw,circle,fill=white, scale=0.5] (A) at (0,10) {};
        \node[draw,circle,fill,scale=0.5] (B) at (1,10) {};
        \node[draw,circle,fill=white,scale=0.5] (C) at (1,6) {};
        \node[draw,circle,fill,scale=0.5] (D) at (3,6) {};
        \node[draw,circle,fill=white,scale=0.5] (E) at (3,3) {};
        \node[draw,circle,fill,scale=0.5] (F) at (6,3) {};
        \node[draw,circle,fill=white,scale=0.5] (G) at (6,1) {};
        \node[draw,circle,fill,scale=0.5] (H) at (10,1) {};
        \node[draw,circle,fill=white,scale=0.5] (I) at (10,0) {};
      
        \path [very thick] [->] (B) edge node[left] {} (A);
        \path [very thick] [->] (B) edge node[left] {} (C);
        \path [very thick] [->] (D) edge node[left] {} (C);
        \path [very thick] [->] (D) edge node[left] {} (E);
        \path [very thick] [->] (F) edge node[left] {} (E);
        \path [very thick] [->] (F) edge node[left] {} (G);
        \path [very thick] [->] (H) edge node[left] {} (G);
        \path [very thick] [->] (H) edge node[left] {} (I);


    \end{tikzpicture}} \\ 
 $\cfk ^{\infty}(T(2,5))$ &$\cfk ^{\infty}(T(5,6))$ 
\end{tabular}
  \caption{The knot complexes for $T(2,5)$ and $T(5,6)$.}
  \label{fig25+56-sep}
\end{figure}
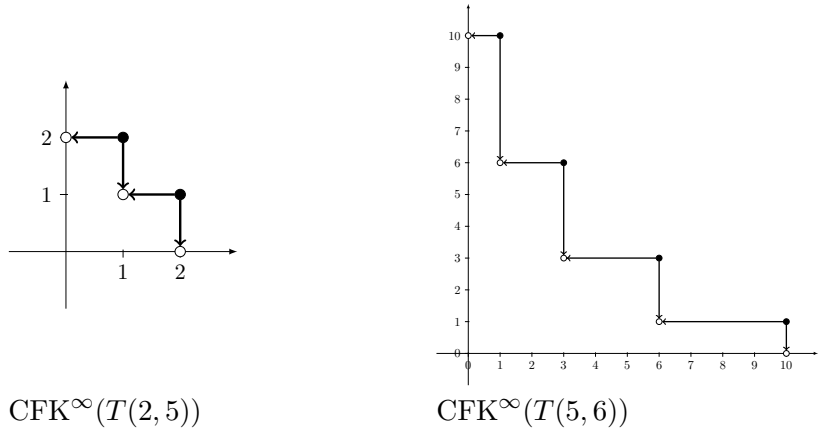

\begin{figure}[ht]
  \centering
  \resizebox{3in}{3in}{     \begin{tikzpicture}
        \coordinate (Origin)   at (0,0);
        \coordinate (XAxisMin) at (-1,0);
        \coordinate (XAxisMax) at (13,0);
        \coordinate (YAxisMin) at (0,-1);
        \coordinate (YAxisMax) at (0,13);
        \draw [thick, black,-latex] (XAxisMin) -- (XAxisMax);
        \draw [thick, black,-latex] (YAxisMin) -- (YAxisMax);
        \clip (-0.9,-0.9) rectangle (12.99,12.99);
       \draw[style=help lines] (-1,-1) grid[step=1cm] (12.95,12.95);
       
       \foreach \x in {0,...,12}
     		    \node[draw=none,fill=none] at (\x+0.5,-0.5) {\x};
    	\foreach \y in {0,...,12}     		
        			\node[draw=none,fill=none] at (-0.5,\y+0.5) {\y};
        

        \node[draw,circle,fill=white, scale=0.5] (A1) at (0.5,12.5) {};
        \node[draw,circle,fill,scale=0.5] (B1) at (1+1/3,12.5) {};
        \node[draw,circle,fill=white,scale=0.5] (C1) at (1+1/3,8.5) {};
        \node[draw,circle,fill,scale=0.5] (D1) at (3+1/3,8.5) {};
        \node[draw,circle,fill=white,scale=0.5] (E1) at (3+1/3,5.5) {};
        \node[draw,circle,fill,scale=0.5] (F1) at (6+1/3,5.5) {};
        \node[draw,circle,fill=white,scale=0.5] (G1) at (6+1/3,3.5) {};
        \node[draw,circle,fill,scale=0.5] (H1) at (10+1/3,3.5) {};
        \node[draw,circle,fill=white,scale=0.5] (I1) at (10+1/3,2.5) {};
        
        \node[draw,circle,fill, scale=0.5] (A2) at (1+2/3,12.7) {};
        \node[draw,circle,fill=gray,scale=0.5] (B2) at (2+2/3,12.7) {};
        \node[draw,circle,fill,scale=0.5] (C2) at (2+2/3,8.7) {};
        \node[draw,circle,fill=gray,scale=0.5] (D2) at (4+2/3,8.7) {};
        \node[draw,circle,fill,scale=0.5] (E2) at (4+2/3,5.7) {};
        \node[draw,circle,fill=gray,scale=0.5] (F2) at (7+2/3,5.7) {};
        \node[draw,circle,fill,scale=0.5] (G2) at (7+2/3,3.7) {};
        \node[draw,circle,fill=gray,scale=0.5] (H2) at (11+2/3,3.7) {};
        \node[draw,circle,fill,scale=0.5] (I2) at (11+2/3,2.5) {};
        
        \node[draw,circle,fill=white, scale=0.5] (A3) at (1+2/3,11.3) {};
        \node[draw,circle,fill,scale=0.5] (B3) at (2+1/3,11.3) {};
        \node[draw,circle,fill=white,scale=0.5] (C3) at (2+1/3,7.3) {};
        \node[draw,circle,fill,scale=0.5] (D3) at (4+1/3,7.3) {};
        \node[draw,circle,fill=white,scale=0.5] (E3) at (4+1/3,4.3) {};
        \node[draw,circle,fill,scale=0.5] (F3) at (7+1/3,4.3) {};
        \node[draw,circle,fill=white,scale=0.5] (G3) at (7+1/3,2.3) {};
        \node[draw,circle,fill,scale=0.5] (H3) at (11+1/3,2.3) {};
        \node[draw,circle,fill=white,scale=0.5] (I3) at (11+1/3,1.3) {};
        
        \node[draw,circle,fill, scale=0.5] (A4) at (2.9,11.6) {};
        \node[draw,circle,fill=gray,scale=0.5] (B4) at (3.6,11.6) {};
        \node[draw,circle,fill,scale=0.5] (C4) at (3.6,7.6) {};
        \node[draw,circle,fill=gray,scale=0.5] (D4) at (5.6,7.6) {};
        \node[draw,circle,fill,scale=0.5] (E4) at (5.6,4.6) {};
        \node[draw,circle,fill=gray,scale=0.5] (F4) at (8.6,4.6) {};
        \node[draw,circle,fill,scale=0.5] (G4) at (8.6,2.6) {};
        \node[draw,circle,fill=gray,scale=0.5] (H4) at (12.6,2.6) {};
        \node[draw,circle,fill,scale=0.5] (I4) at (12.6,1.6) {};
        
        \node[draw,circle,fill=white, scale=0.5] (A5) at (2.9,10.5) {};
        \node[draw,circle,fill,scale=0.5] (B5) at (3.9,10.5) {};
        \node[draw,circle,fill=white,scale=0.5] (C5) at (3.9,6.5) {};
        \node[draw,circle,fill,scale=0.5] (D5) at (5.9,6.5) {};
        \node[draw,circle,fill=white,scale=0.5] (E5) at (5.9,3.1) {};
        \node[draw,circle,fill,scale=0.5] (F5) at (8.9,3.1) {};
        \node[draw,circle,fill=white,scale=0.5] (G5) at (8.9,1.1) {};
        \node[draw,circle,fill,scale=0.5] (H5) at (12.2,1.1) {};
        \node[draw,circle,fill=white,scale=0.5] (I5) at (12.2,0.5) {};

        \path [very thick] [->] (B1) edge node[left] {} (A1);
        \path [very thick] [->] (B1) edge node[left] {} (C1);
        \path [very thick] [->] (D1) edge node[left] {} (C1);
        \path [very thick] [->] (D1) edge node[left] {} (E1);
        \path [very thick] [->] (F1) edge node[left] {} (E1);
        \path [very thick] [->] (F1) edge node[left] {} (G1);
        \path [very thick] [->] (H1) edge node[left] {} (G1);
        \path [very thick] [->] (H1) edge node[left] {} (I1);
        
        \path [very thick] [->] (B2) edge node[left] {} (A2);
        \path [very thick] [->] (B2) edge node[left] {} (C2);
        \path [very thick] [->] (D2) edge node[left] {} (C2);
        \path [very thick] [->] (D2) edge node[left] {} (E2);
        \path [very thick] [->] (F2) edge node[left] {} (E2);
        \path [very thick] [->] (F2) edge node[left] {} (G2);
        \path [very thick] [->] (H2) edge node[left] {} (G2);
        \path [very thick] [->] (H2) edge node[left] {} (I2);
        
        \path [very thick] [->] (B3) edge node[left] {} (A3);
        \path [very thick] [->] (B3) edge node[left] {} (C3);
        \path [very thick] [->] (D3) edge node[left] {} (C3);
        \path [very thick] [->] (D3) edge node[left] {} (E3);
        \path [very thick] [->] (F3) edge node[left] {} (E3);
        \path [very thick] [->] (F3) edge node[left] {} (G3);
        \path [very thick] [->] (H3) edge node[left] {} (G3);
        \path [very thick] [->] (H3) edge node[left] {} (I3);
        
        \path [very thick] [->] (B4) edge node[left] {} (A4);
        \path [very thick] [->] (B4) edge node[left] {} (C4);
        \path [very thick] [->] (D4) edge node[left] {} (C4);
        \path [very thick] [->] (D4) edge node[left] {} (E4);
        \path [very thick] [->] (F4) edge node[left] {} (E4);
        \path [very thick] [->] (F4) edge node[left] {} (G4);
        \path [very thick] [->] (H4) edge node[left] {} (G4);
        \path [very thick] [->] (H4) edge node[left] {} (I4);
        
        \path [very thick] [->] (B5) edge node[left] {} (A5);
        \path [very thick] [->] (B5) edge node[left] {} (C5);
        \path [very thick] [->] (D5) edge node[left] {} (C5);
        \path [very thick] [->] (D5) edge node[left] {} (E5);
        \path [very thick] [->] (F5) edge node[left] {} (E5);
        \path [very thick] [->] (F5) edge node[left] {} (G5);
        \path [very thick] [->] (H5) edge node[left] {} (G5);
        \path [very thick] [->] (H5) edge node[left] {} (I5);
        
        \path [very thin] [->] (A2) edge node[left] {} (A1);
        \path [very thin] [->] (A2) edge node[left] {} (A3);
        \path [very thin] [->] (A4) edge node[left] {} (A3);
        \path [very thin] [->] (A4) edge node[left] {} (A5);
        
        \path [very thin] [->] (B2) edge node[left] {} (B1);
        \path [very thin] [->] (B2) edge node[left] {} (B3);
        \path [very thin] [->] (B4) edge node[left] {} (B3);
        \path [very thin] [->] (B4) edge node[left] {} (B5);

        \path [very thin] [->] (C2) edge node[left] {} (C1);
        \path [very thin] [->] (C2) edge node[left] {} (C3);
        \path [very thin] [->] (C4) edge node[left] {} (C3);
        \path [very thin] [->] (C4) edge node[left] {} (C5);

        \path [very thin] [->] (D2) edge node[left] {} (D1);
        \path [very thin] [->] (D2) edge node[left] {} (D3);
        \path [very thin] [->] (D4) edge node[left] {} (D3);
        \path [very thin] [->] (D4) edge node[left] {} (D5);

        \path [very thin] [->] (E2) edge node[left] {} (E1);
        \path [very thin] [->] (E2) edge node[left] {} (E3);
        \path [very thin] [->] (E4) edge node[left] {} (E3);
        \path [very thin] [->] (E4) edge node[left] {} (E5);
                
        \path [very thin] [->] (F2) edge node[left] {} (F1);
        \path [very thin] [->] (F2) edge node[left] {} (F3);
        \path [very thin] [->] (F4) edge node[left] {} (F3);
        \path [very thin] [->] (F4) edge node[left] {} (F5);

        \path [very thin] [->] (G2) edge node[left] {} (G1);
        \path [very thin] [->] (G2) edge node[left] {} (G3);
        \path [very thin] [->] (G4) edge node[left] {} (G3);
        \path [very thin] [->] (G4) edge node[left] {} (G5);
        
        \path [very thin] [->] (H2) edge node[left] {} (H1);
        \path [very thin] [->] (H2) edge node[left] {} (H3);
        \path [very thin] [->] (H4) edge node[left] {} (H3);
        \path [very thin] [->] (H4) edge node[left] {} (H5);
        
        \path [very thin] [->] (I2) edge node[left] {} (I1);
        \path [very thin] [->] (I2) edge node[left] {} (I3);
        \path [very thin] [->] (I4) edge node[left] {} (I3);
        \path [very thin] [->] (I4) edge node[left] {} (I5);


    \end{tikzpicture}} \\ 
  \caption{The knot complex for $T(2,5)\cs T(5,6)$.}
  \label{fig25+56-tensor}
\end{figure}

Considering lines of slope $-\frac{3}{2}$ (corresponding to $t=\frac{4}{5}$), an analysis of the bifiltered basis elements in $\cfk ^{\infty} (T(2,5)\cs T(5,6))$ reveals that at $j$--intercept $b=\frac{19}{2}$ corresponding to $s=\frac{19}{5}$, and for no smaller $b$ or $s$, the line $\call_{\frac{4}{5},s}$ contains basis elements.  In fact, the line contains exactly two bifiltered basis elements -- those represented by $(1,8)$ and $(3,5)$ in Figure ~\ref{fig25+56-tensor} and arising from the tensor product elements $(0,2)\otimes (1,6)$ and $(0,2)\otimes (3,3)$ respectively.  Denote by $A$ the element represented by $(1,8)$ and $B$ the element represented by $(3,5)$.

Note that Theorem ~\ref{feller-krcatovich} implies that the singularities of $\U'_{T(5,7)}(t)$ and $\U'_{T(2,5)\cs T(5, 6)}(t)$ occur at the same $t$--values.  So $\frac{4}{5}$ is a singularity of $\U'$ and we compute the secondary Upsilon invariant at $t_0=\frac{4}{5}$.  Now, suppose that $A$ and $B$ are homologous in $\calf_{\frac{4}{5},\frac{23}{5}}$.  Then, since both $A$ and $B$ are at Maslov  grading 0, there would be basis elements $x_1, x_2, ... \,, x_k$ in $\calf_{\frac{4}{5},\frac{23}{5}}$ at Maslov grading 1 such that 
\begin{equation}
\partial(b_1x_1+b_2x_2+\cdots +b_kx_k)=A+B
\label{boundary}
\end{equation}
for some $b_i\in\Z_2$.  We compute for all basis elements in $\cfk ^{\infty}(T(2,5)\cs T(5,6))$  which are at Maslov grading 1 (note that these must be tensor products of one element at grading 0 and one at grading 1), the value of $s$ for which the element is on the line $\call_{\frac{4}{5},s}$.  See Figure ~\ref{tensor table} for the full list of computations.
\begin{figure}[ht]
\begin{tabular}{| c || c | c || c |}
\hline
\rule{0pt}{1.02\normalbaselineskip} \rule[-1.2ex]{0pt}{0pt}$\cfk ^{\infty}(T(2,5))_0\otimes \cfk ^{\infty}(T(5,6))_1$& $i$&$j$ &$s = \frac{2}{5}j+\frac{3}{5}i$\\
\hline
\rule{0pt}{1.02\normalbaselineskip}$(0,2)\otimes(1,10)$ & 1&12 & $27/5$\\
$(0,2)\otimes(3,6)$& 3 & 8 & $25/5$\\
$(0,2)\otimes(6,3)$& 6&5 & $28/5$\\
$(0,2)\otimes(10,1)$& 10&3 & $36/5$\\
$(1,1)\otimes(1,10)$& 2&11& $28/5$\\
$(1,1)\otimes(3,6)$& 4&7& $26/5$\\
$(1,1)\otimes(6,3)$& 7&4 & $29/5$\\
$(1,1)\otimes(10,1)$& 11&2 & $37/5$\\
$(2,0)\otimes(1,10)$& 3&10 & $29/5$\\
$(2,0)\otimes(3,6)$& 5&6& $27/5$\\
$(2,0)\otimes(6,3)$& 8&3& $30/5$\\
$(2,0)\otimes(10,1)$& 12&1 & $38/5$\\
\hline
\hline
\rule{0pt}{1.02\normalbaselineskip} \rule[-1.2ex]{0pt}{0pt}$\cfk ^{\infty}(T(2,5))_1\otimes \cfk ^{\infty}(T(5,6))_0$& $i$ & $j$ &$s = \frac{2}{5}j+\frac{3}{5}i$\\
\hline
\rule{0pt}{1.02\normalbaselineskip}$(1,2)\otimes (0,10)$ & 1&12 & $27/5$ \\
$(1,2)\otimes (1,6)$ & 2&8 & $22/5$ \\
$(1,2)\otimes (3,3)$ & 4&5 & $22/5$ \\
$(1,2)\otimes (6,1)$ & 7&3 & $27/5$ \\
$(1,2)\otimes (10,0)$ & 11&2 & $37/5$ \\
$(2,1)\otimes (0,10)$ & 2&11 & $28/5$ \\
$(2,1)\otimes (1,6)$ & 3&7 & $23/5$ \\
$(2,1)\otimes (3,3)$ & 5&4 & $23/5$ \\
$(2,1)\otimes (6,1)$ & 8&2 & $28/5$ \\
$(2,1)\otimes (10,0)$ & 12&1 & $38/5$ \\
\hline
\end{tabular}
\caption{The value of $s$ for which each element at grading 1 is on the line $\call_{\frac{4}{5},s}$.}
\label{tensor table}
\end{figure}

  Our search results in exactly four elements within the desired range:
$$x_1=(1,2)\otimes (1,6), x_2=(1,2)\otimes (3,3), x_3=(2,1)\otimes (1,6), \text{ and } x_4=(2,1)\otimes (3,3).$$
Taking the boundaries, we get:
\begin{gather*}
\partial (x_1)=((0,2)+(1,1))\otimes (1,6) = (0,2)\otimes (1,6)+(1,1)\otimes (1,6)=A+(1,1)\otimes (1,6),\\ 
\partial (x_2)=((0,2)+(1,1))\otimes (3,3)=(0,2)\otimes (3,3)+(1,1)\otimes (3,3)=B+(1,1)\otimes (3,3),\\
\partial (x_3)=((2,0)+(1,1))\otimes (1,6)=(2,0)\otimes (1,6)+(1,1)\otimes (1,6),\\
 \partial (x_4)=((2,0)+(1,1))\otimes (3,3)=(2,0)\otimes (3,3)+(1,1)\otimes (3,3).
\end{gather*}
Notice that if Equation ~\ref{boundary} is to hold, it must be that $b_1=b_2=1$.  Since $$\partial(x_1+x_2)=A+B+(1,1)\otimes (1,6)+(1,1)\otimes (3,3),$$ we need $b_3=b_4=1$ in order to counteract the extra contributions of $x_1$ and $x_2$. However, $$\partial(x_1+x_2+x_3+x_4)=A+B+(2,0)\otimes (1,6)+(2,0)\otimes (3,3)$$ and we are left without options.  So $A$ and $B$ are not homologous in $\calf_{\frac{4}{5},\frac{23}{5}}$.  Thus $$\gamma^2_{T(2,5)\cs T(5,6), \frac{4}{5}}\left(\frac{4}{5}\right) > \frac{23}{5}$$ and so
$$\U^2_{T(2,5)\cs T(5,6),\frac{4}{5}}(\frac{4}{5})=-2\cdot\left(\gamma^2_{T(2,5)\cs T(5,6), \frac{4}{5}}\left(\frac{4}{5}\right)-\frac{19}{5}\right)<-2\cdot \left(\frac{23}{5}-\frac{19}{5}\right)=-\frac{8}{5}.$$
\end{proof}
\noindent Thus, since $\U^2$ is a stable equivalence invariant, it follows that $\cfk ^{\infty}(T(5,7))$ is not stably equivalent to $\cfk ^{\infty}(T(2,5)\cs  T(5,6))$.  This completes the proof of Theorem ~\ref{main thm}.

\subsection{Proof of the general theorem}
$ $ \newline The proof of Theorem \ref{general thm} follows similar steps to those of Theorem \ref{main thm}.  We will:
\begin{enumerate}
\item Construct the knot complex for $T(p,\,p+2)$ \vspace{0.05 in}
\item Compute $\U^2_{T(p,\,p+2),\frac{4}{p}}(\frac{4}{p})$
\item Construct the knot complex for $T(2,\,p)\cs T(p, \,p+1)$ \vspace{0.05 in}
\item Show that $\U^2_{T(2,\,p)\cs T(p, \,p+1),\frac{4}{p}}(\frac{4}{p}) <\U^2_{T(p,\,p+2),\frac{4}{p}}(\frac{4}{p})$
\end{enumerate}
In steps $(1)$ and $(3)$ we will use the relationship between the semigroup generated by $p,q$ and the Alexander polynomial  $\Delta_{T(p, \,q)}(t)$ given in \cite{borodzik-livingston}:
\begin{equation}
\frac{\Delta_{T(p,q)}(t)}{1-t} = \sum_{s\in S_{p,q}}t^s.
\label{semigroup alex}
\end{equation}
This relationship combined with the method given in Section \ref{basics} describes the staircase.

In step $(2)$ we show that $t_0=\frac{4}{p}$ is a singularity of $\U'_{T(p,\,p+2)}(t)$ and identify the two pivot points in the complex at this $t$--value.  Then we compute $\U^2_{T(p,\,p+2),\frac{4}{p}}(\frac{4}{p})$ from the staircase complex by showing that the two pivot points become homologous in $\calf_{\frac{4}{p}, \frac{p^2+p-7}{p}}$. 

Finally, in step $(4)$, we see that, as in step (2), $t_0=\frac{4}{p}$ is a singularity of $\U'_{T(2,\,p)\cs T(p,\,p+1)}(t)$ and there are precisely two bifiltered basis elements acting as pivot points for $\U_{T(2,\,p)\cs T(p, \,p+1)}(t)$ at this $t$--value.  In order to show that these two elements do not become homologous in $\calf_{\frac{4}{p}, \frac{p^2+p-7}{p}}$, we compute that, as in the proof of Theorem \ref{main thm}, there are precisely four bifiltered basis elements at Maslov grading 1 in this subcomplex and no combination of the four has boundary equal to the sum of the pivot points.

\begin{proof}
$ $\newline
\textbf{Step (1):} Let $S_{p,\,q}$ be the semigroup generated by $p$ and $q$, i.e., $S_{p,\,q}=\{np+mq\;|\;n,m\in \Z_{\geq 0}\}.$
We have that 
\begin{equation}\arraycolsep=1.4pt\def\arraystretch{1}
\begin{array}{rll}
S_{p,\,p+2} =\{&0,&\\
 &p, p+2&\\
&2p, 2p+2, 2p+4,&\\
 &3p, 3p+2, 3p+4, 3p+6,&\\
&\vdots&\\
&np, np+2, np+4, ... , np+2n,&\\
&\vdots&\\
&(p-1)p, (p-1)p+2, ... , (p-1)p+2(p-1)\;\}\cup \Z_{\geq (p-1)(p+1)}.&
\end{array}
\label{S_p,p+2}
\end{equation}
The following is a relationship between the Alexander polynomial of $T(p,q)$ and its semigroup, given in \cite{borodzik-livingston},
$$\frac{\Delta_{T(p,q)}(t)}{1-t} = \sum_{s\in S_{p,\,q}}t^s,$$
in other words, 
$$\Delta_{T(p,q)}(t)=\sum_{s\in S_{p,\,q}}t^s-t^{s+1}.$$
Since $T(p, \, p+2)$ is an $L$--space knot, $\cfk^{\infty}(T(p, \, p+2))$ is then a staircase of the form $$[a_1-a_0, a_2-a_1,...,a_d-a_{d-1}]$$ where $d=p^2-1$ and
$$\Delta_{T(p, \, p+2)}(t) = \sum_{i=0}^d (-1)^it^{a_i}.$$

Order the elements in the semigroup $S_{p,\,p+2}=\{s_0, s_1, s_2, ...\}$ such that $s_i<s_{i+1}$.  Note that $S_{p,\,p+2}$ as shown in (\ref{S_p,p+2}) is in increasing order through the element $s_{i^*}=\left( \frac{p+1}{2} \right) p+2\left(\frac{p-1}{2}\right)$, and for $s_i\leq s_{i^*}$, $s_i-s_{i-1}>1$.  So we have that $\cfk^{\infty}(T(p, \, p+2))$ is a staircase with initial portion:
$$[(s_0+1)-s_0, s_1-(s_0+1), (s_1+1)-s_1, s_2-(s_1+1),..., (s_{{i^*}-1}+1)-s_{{i^*}-1}, s_{i^*}-(s_{{i^*}-1}+1)]$$
\begin{equation}
=[1, s_1-(s_0+1), 1, s_2-(s_1+1),..., 1, s_{i^*}-(s_{{i^*}-1}+1)].
\label{staircase}
\end{equation}
On the one hand, adding the first steps through $s_i^*$, we have
$$\sum_{s_i\leq s_{i^*}} 1+s_i-(s_{i-1}+1) = s_{i^*},$$
and on the other hand, adding the first $d/2$ steps
$$d/2 = \sum_{i=1}^{d/2} a_i-a_{i-1}$$
by symmetry of the  $\cfk^{\infty}(T(p, \, p+2))$ staircase.  Since $s_i^*>d/2$, this implies that the full staircase is (\ref{staircase}), where the pattern is truncated and reflected after the $(d/2)$th term.  More precisely, $\cfk^{\infty}(T(p, \, p+2))$ is a staircase of the form
$$[1, p-1,1,1,1,p-3, 1,1,1,1,1,p-5, ... ,\underbrace{1,1, ... ,1}_{2j+1} , p-(2j+1), ... ]$$
where the pattern is truncated and reflected after the $(p^2-1)/2$th term.  This gives us a bifiltered basis $\mathcal{B}$ for $\cfk^{\infty}(T(p, \, p+2))$.

\textbf{Step (2):} Note that the points $$A=\left(1, \frac{(p-1)(p+1)}{2}-(p-1)\right) \text{ and } B=\left(3, \frac{(p-1)(p+1)}{2}-(p-1)-1-(p-3)\right)$$ both lie on the line $\mathcal{L}_{\frac{4}{p},\frac{p^2-p-1}{p}}$ given by
$$\, j=-\frac{p-2}{2}i+\frac{p^2-p-1}{2}.$$
A computation shows that all other points in the diagram of $\cfk^{\infty}(T(p, \, p+2))$ are above this line.  Thus $\gamma\,_{T(p, \, p+2)}(\frac{4}{p})=\frac{p^2-p-1}{p}$.  So near $t=\frac{4}{p}$, the line $\call_{t,s}$ pivots around the two points $A$ and $B$.  This causes a change in slope in $\U_{T(p, \, p+2)}$ and so $t=\frac{4}{p}$ is a singulariy of $\U_{T(p, \, p+2)}'$ and $\calc_{\frac{4}{p}}=\{A,B\}$.  

We now compute $\U^2_{T(p,\,p+2),\frac{4}{p}}(\frac{4}{p})$. For small enough $\delta$,
$$\calc_{\frac{4}{p}-\delta}=\{A\} \text{ and } \calc_{\frac{4}{p}+\delta}=\{B\}.$$
To determine  $\U^2_{T(p, \, p+2),\frac{4}{p}}(\frac{4}{p})$, we compute how far the line of slope $-\frac{(p-2)}{2}$ needs to be moved so that the elements represented by $A$ and $B$ are homologous in $\calf_{\frac{4}{p},r}$. 

  Based on the staircase, we see that we need $\calf_{\frac{4}{p},r}$ to contain the basis elements represented by $A+(1,0)$ and $A+(2,-1)$.  The minimal $r$ which accomplishes this is $r=\frac{p^2+p-7}{p}$.  So we have that $\gamma^2_{T(p, \, p+2),\frac{4}{p}}(\frac{4}{p}) = \frac{p^2+p-7}{p}$.  Thus 
$$\U^2_{T(p, \, p+2),\frac{4}{p}}(\frac{4}{p}) = -2\cdot \left(\frac{p^2+p-7}{p}-\frac{p^2-p-1}{p}\right) = -4\cdot \frac{p-3}{p}.$$


\textbf{Step (3):} The diagrams for the chain complexes $\cfk ^{\infty}(T(2,p))$ and $\cfk ^{\infty}(T(p, \, p+1))$ are computed from their semigroups.  We have that
$$S_{2,p}=\{0,2,4, ... , p-1\}\cup\Z_{\geq p},$$
so
$$\Delta_{T(2,p)} = 1-t+t^2-t^3+t^4-t^5+\cdots +t^{p-1}$$
thus the staircase for $\cfk ^{\infty}(T(2,p))$ is 
$$[\underbrace{1,1,1,1, ... , 1}_{p-1}].$$
Similarly, 
\begin{equation*}\arraycolsep=1.4pt\def\arraystretch{1}
\begin{array}{rll}
S_{p,\,p+1} =\{&0,p, p+1,2p, 2p+1, 2p+2, ... ,&\\
&(p-2)p, (p-2)p+1, ... , (p-2)p+(p-2)\;\}\cup \Z_{\geq (p-1)p}&
\end{array}
\end{equation*}
so
$$\Delta_{T(p, \, p+1)} = 1-t+t^p-t^{p+2}+t^{2p}-t^{2p+3}+\cdots  + t^{(p-2)p}-t^{(p-2)p+(p-1)}+t^{(p-1)p},$$
and thus the staircase for $\cfk ^{\infty}(T(p, \, p+1))$ is 
$$[1, p-1, 2, p-2, ... , j, p-j, ..., p-1, 1].$$

From these staircase descriptions, a bifiltered basis $\calb_{2,\,p} = \{\alpha_i\}_{i=0}^p$, $\calb_{p,\,p+1} = \{\beta_i\}_{i=0}^{2p-1}$ for each complex can be determined:
$$\alpha_{2i} \text{ is represented by } \left(i, \frac{p-1}{2}-i\right),$$
$$\alpha_{2i+1} \text{ is represented by } \left(i+1, \frac{p-1}{2}-i\right),$$
$$\beta_{2i} \text{ is represented by } \left(\sum_{n=1}^i n, \frac{(p-1)p}{2}-\sum_{n=1}^i (p-n)\right),$$
$$\beta_{2i+1} \text{ is represented by } \left(\sum_{n=1}^{i+1} n, \frac{(p-1)p}{2}-\sum_{n=1}^i (p-n)\right).$$
Here even-indexed elements are at Maslov grading 0, while odd-indexed elements are at Maslov grading 1.  A bifiltered basis for the tensor product is the tensor product of the bases $\calb_{2,\,p}\otimes\calb_{p,\,p+1}=\{\alpha_i\otimes\beta_j\}$.

\textbf{Step (4):} In the tensor product, the points $\alpha_0\otimes \beta_2$ and $\alpha_0\otimes \beta_4$ are at the same filtration levels as $A$ and $B$ respectively.  Thus they lie on a line of slope $-\frac{p-2}{2}$ (corresponding to $t=\frac{4}{p}$) with $j$--intercept $\frac{p^2-p-1}{2}$ (corresponding to $s=\frac{p^2-p-1}{p}$).  We need to confirm that all other bifiltered basis elements in the tensor prodcut lie above this line.

First, note that $\alpha_{2i+1}\otimes \beta_{2j}$, $\alpha_{2i+1}\otimes \beta_{2j+1}$, and $\alpha_{2i}\otimes \beta_{2j+1}$ are at higher filtration levels than $\alpha_{2i}\otimes \beta_{2j}$.  So we will show that for all $(i,j)\neq (0,1)\text{ or } (0,2)$, $\alpha_{2i}\otimes \beta_{2j}$ is above line $\call_{\frac{4}{p},\frac{p^2-p-1}{p}}$ given by  
$$y=-\frac{p-2}{2}\,x+\frac{p^2-p-1}{2}.$$
 The element $\alpha_{2i}\otimes \beta_{2j}$ is represented by
$$\left(i+\sum_{n=1}^j n, \frac{p-1}{2}-i+\frac{(p-1)p}{2}-\sum_{n=1}^j (p-n) \right)=\left(i+\frac{j(j+1)}{2}, \frac{p^2-1}{2}-i-jp+\frac{j(j+1)}{2}\right).$$
We test the inequality
$$y\leq-\frac{p-2}{2}x+\frac{p^2-p-1}{2},$$
at the $x$-- and $y$--values above and find that
$$ \frac{p^2-1}{2}-i-jp+\frac{j(j+1)}{2}\leq-\frac{p-2}{2}\left(i+\frac{j(j+1)}{2}\right)+\frac{p^2-p-1}{2},$$
$$-i-jp+\frac{j^2+j}{2}\leq -\frac{p-2}{2}i-\frac{p-2}{2}\cdot \frac{j^2+j}{2}-\frac{p}{2},$$
$$\frac{p-4}{2} i \leq -\frac{p}{2}\cdot\frac{j^2+j}{2}+jp-\frac{p}{2,}$$
$$\frac{p-4}{2}i\leq -\frac{p}{4}j^2+\frac{3p}{4}j-\frac{p}{2},$$
$$i\leq \frac{2}{p-4} \left(-\frac{p}{4}j^2+\frac{3p}{4}j-\frac{p}{2}\right),$$
$$i\leq -\frac{p}{2(p-4)} (j^2 -3j+2),$$
\begin{equation} i\leq - \frac{p}{2(p-4)} (j-2)(j-1). \label{ineq1} \end{equation}
Inequality \ref{ineq1} holds only for $i=0$ and $j=1$ or $2$.  So for all other values of $i$ and $j$, $\alpha_{2i}\otimes\beta_{2j}$ is above the line $\call_{\frac{4}{p},\frac{p^2-p-1}{p}}$.

Theorem ~\ref{feller-krcatovich} implies that the singularities of $\U'_{T(p, \, p+2)}(t)$ and $\U'_{T(2,p)\cs T(p, \, p+1)}(t)$ occur at the same $t$--values.  Thus $\frac{4}{p}$ is a singularity of $\U'_{T(2,p)\cs T(p, \, p+1)}$ and so we can consider the secondary Upsilon invariant of $T(2,p)\cs T(p, \, p+1)$ at $t_0=\frac{4}{p}$.  Now, suppose that $\alpha_0\otimes\beta_2$ and $\alpha_0\otimes\beta_4$ are homologous in $\calf_{\frac{4}{p},\frac{p^2+p-7}{p}}$.  Then, since both $\alpha_0\otimes\beta_2$ and $\alpha_0\otimes\beta_4$ are at Maslov  grading 0, there would be basis elements $x_1, x_2, ... \,, x_k$ in $\calf_{\frac{4}{p},\frac{p^2+p-7}{p}}$ at Maslov grading 1 such that 
\begin{equation}
\partial(b_1x_1+b_2x_2+\cdots +b_kx_k)= \alpha_0\otimes\beta_2+\alpha_0\otimes\beta_4
\label{boundary2}
\end{equation}
for some $b_i\in\Z_2$.   Bifiltered basis elements at Maslov grading 1 have the form

\begin{align*}
\alpha_{2i}\otimes\beta_{2j+1} &= \left(i+\sum_{n=1}^{j+1}n\, , \,\frac{p-1}{2}-i+\frac{(p-1)p}{2}-\sum_{n=1}^j(p-n)  \right)\\
&=\left(i+\frac{(j+1)(j+2)}{2} \, , \, \frac{p^2-1}{2} -i-jp+\frac{j(j+1)}{2}\right)
\end{align*}
or
\begin{align*}
\alpha_{2i+1}\otimes\beta_{2j} &= \left(i+1+\sum_{n=1}^{j}n\, , \,\frac{p-1}{2}-i+\frac{(p-1)p}{2}-\sum_{n=1}^j(p-n)  \right)\\
&=\left(i+1+\frac{j(j+1)}{2} \, , \, \frac{p^2-1}{2} -i-jp+\frac{j(j+1)}{2}\right).
\end{align*}
To determine which elements of Maslov grading 1 are in  $\calf_{\frac{4}{p},\frac{p^2+p-7}{p}}$, we determine which of the above satisfy the inequality 
$$y\leq -\frac{p-2}{2}\,x+\frac{p^2+p-7}{2}.$$
For $\alpha_{2i}\otimes\beta_{2j+1}$ we have
\begin{align*}
\frac{p^2-1}{2} -i-jp+\frac{j(j+1)}{2} &\leq -\frac{p-2}{2}\left(i+\frac{(j+1)(j+2)}{2}\right) + \frac{p^2+p-7}{2} \\
\frac{p-4}{2}i &\leq -\frac{p-2}{2}\cdot \frac{j^2+3j+2}{2} - \frac{j^2+j}{2} +jp +\frac{p-6}{2}\\
\frac{p-4}{2}i&\leq -\frac{p}{4}j^2+\frac{p+4}{4}j-2 \\
i&\leq \frac{2}{p-4} \left(-\frac{p}{4}j^2+\frac{p+4}{4}j-2 \right) \\
i&\leq -\frac{p}{2(p-4)} \left(j^2-\left(1+\frac{4}{p}\right)j+\frac{8}{p} \right).
\end{align*}
Since the right-hand side of the final inequality is negative for $j\geq 0$, the element $\alpha_{2i}\otimes\beta_{2j+1}$ is not in  $\calf_{\frac{4}{p},\frac{p^2+p-7}{p}}$ for any $i, j$.
For $\alpha_{2i+1}\otimes\beta_{2j}$ we have
\begin{align*}
\frac{p^2-1}{2} -i-jp+\frac{j(j+1)}{2} &\leq -\frac{p-2}{2}\left(i+1+\frac{j(j+1)}{2}\right) + \frac{p^2+p-7}{2} \\
\frac{p-4}{2}i &\leq -\frac{p-2}{2}\cdot \frac{j^2+j+2}{2} - \frac{j^2+j}{2}+jp +\frac{p-6}{2}\\
\frac{p-4}{2}i&\leq -\frac{p}{4}j^2+\frac{3p}{4}j-2 \\
i&\leq \frac{2}{p-4} \left(-\frac{p}{4}j^2+\frac{3p}{4}j-2 \right) \\
i&\leq -\frac{p}{2(p-4)} \left(j^2-3j+\frac{8}{p} \right) \\
i&\leq -\frac{p}{2(p-4)}\left(j-\frac{3+\sqrt{9-\frac{32}{p}}}{2} \right) \left(j-\frac{3-\sqrt{9-\frac{32}{p}}}{2} \right).
\end{align*}
This inequality only holds when
$$\frac{3-\sqrt{9-\frac{32}{p}}}{2}\leq j\leq \frac{3+\sqrt{9-\frac{32}{p}}}{2} \text{ and } 0\leq i\leq -\frac{p}{2(p-4)}\left(\left(\frac{3}{2}\right)^2-3\cdot \frac{3}{2}+\frac{8}{p}\right),$$
which is when
$$1\leq j \leq 2 \text{ and } 0\leq i\leq 1.$$Thus there are four elements within the desired range:
$$\alpha_1\otimes\beta_2,\, \alpha_1\otimes\beta_4,\, \alpha_3\otimes\beta_2, \text{ and } \alpha_3\otimes\beta_4.$$
Taking the boundaries, we get:
\begin{align*}
\partial (\alpha_1\otimes\beta_2) &= \partial \alpha_1 \otimes \beta_2 +\alpha_1 \otimes \partial \beta_2\\
&= (\alpha_0 +\alpha_2) \otimes \beta_2 + \alpha_1 \otimes 0 \\
& =  \alpha_0 \otimes \beta_2 + \alpha_2 \otimes \beta_2, \\
\partial (\alpha_1\otimes\beta_4) &= \partial \alpha_1 \otimes \beta_4 +\alpha_1 \otimes \partial \beta_4\\
&= (\alpha_0 +\alpha_2) \otimes \beta_4 + \alpha_1 \otimes 0 \\
& =  \alpha_0 \otimes \beta_4 + \alpha_2 \otimes \beta_4, \\
\partial (\alpha_3\otimes\beta_2) &= \partial \alpha_3 \otimes \beta_2 +\alpha_3 \otimes \partial \beta_2\\
&= (\alpha_2 +\alpha_4) \otimes \beta_2 + \alpha_3 \otimes 0 \\
& =  \alpha_2 \otimes \beta_2 + \alpha_4 \otimes \beta_2, \\
\partial (\alpha_3\otimes\beta_4) &= \partial \alpha_3 \otimes \beta_4 +\alpha_3 \otimes \partial \beta_4\\
&= (\alpha_2 +\alpha_4) \otimes \beta_4 + \alpha_3 \otimes 0 \\
& =  \alpha_2 \otimes \beta_4 + \alpha_4 \otimes \beta_4. \\
\end{align*}
Notice that if Equation ~\ref{boundary2} is to hold, it must be that $\alpha_1\otimes\beta_2$ and  $\alpha_1\otimes\beta_4$ have coefficients of 1.  Since 
$$\partial(\alpha_1\otimes\beta_2 + \alpha_1\otimes\beta_4)=\alpha_0\otimes\beta_2+\alpha_0\otimes\beta_4+\alpha_2\otimes\beta_2 + \alpha_2 \otimes\beta_4,$$
it must be that  $\alpha_3\otimes\beta_2$ and  $\alpha_3\otimes\beta_4$ also have coefficients of 1. However, $$\partial(\alpha_1\otimes\beta_2+\alpha_1\otimes\beta_4+\alpha_3\otimes\beta_2 + \alpha_3 \otimes\beta_4)=\alpha_0\otimes\beta_2+\alpha_0\otimes\beta_4+\alpha_4\otimes\beta_2 + \alpha_4 \otimes\beta_4$$ and we are left without options.  So $\alpha_0\otimes\beta_2$ and $\alpha_0\otimes\beta_4$ are not homologous in $\calf_{\frac{4}{5},\frac{23}{5}}$.  Thus $$\gamma^2_{T(2,p)\cs T(p, \, p+1), \frac{4}{p}}\left(\frac{4}{p}\right) > \frac{p^2+p-7}{p}$$ and so
\begin{align*} 
\U^2_{T(2, p)\cs T(p, \, p+1),\frac{4}{p}}\left(\frac{4}{p}\right)&=-2\cdot\left(\gamma^2_{T(2, p)\cs T(p, \, p+1), \frac{4}{p}}\left(\frac{4}{p}\right)-\frac{p^2-p-1}{p}\right)\\
&<-2\cdot \left(\frac{p^2+p-7}{p}-\frac{p^2-p-1}{p}\right)\\
&=-4\cdot \frac{p-3}{p}
\end{align*}
as desired.

Since $\U^2$ is a stable equivalence invariant, it follows that $\cfk ^{\infty}(T(p, \, p+2))$ is not stably equivalent to $\cfk ^{\infty}(T(2, p)\cs  T(p, \, p+1))$.
\end{proof}

It may be that steps similar to those of the proof of  Theorem \ref{general thm} can be used to generalize it. The following is a conjecture of the author.
\begin{conjecture}
For all $p\geq 5$ and $2\leq k \leq p-2$ such that gcd($p, k)=1$, the knot complex $\cfk ^{\infty}(T(p, \, p+k))$ is not stably equivalent to $\cfk ^{\infty}(T(k, p)\cs T(p, \, p+1))$.
\end{conjecture}
\noindent Note that the Feller-Krcatovich relationships among the Upsilon functions of torus knots do extend to stable equivalence in some cases.  For example, with a change of basis one can see that the knot complexes $\cfk ^{\infty}(T(2,3) \cs T(2,3))$ and $\cfk ^{\infty}(T(2,5))$ are stably equivalent.  In a recent paper, Kim, Krcatovich, and Park \cite{kim-krcatovich-park} gave a condition for the knot complex of the connected sum of two $L$--space knots to be stably equivalent to a staircase complex.  Using this result, we can see that $\cfk ^{\infty}(T(p-1, \, p) \cs T(p, \, p+1))$ and $\cfk ^{\infty}(T(p, \,2p-1))$ are stably equivalent.  As a result, we limit our conjecture to $k\leq p-2$ and $p\geq 5$.


\raggedright
\bibliography{mybiblio}
\bibliographystyle{abbrv}

\end{document}